\documentclass[11pt,fleqn]{article}

\usepackage{amsfonts}
\usepackage{amsmath}
\usepackage{amsthm}
\usepackage{bm}
\usepackage{cases}
\usepackage{cite}
\usepackage{graphicx}
\usepackage{indentfirst}
\usepackage{mathrsfs}
\usepackage{titlesec}
\usepackage{makecell}
\usepackage{rotating}
\usepackage[T1]{fontenc}

\allowdisplaybreaks

\numberwithin{equation}{section}

\newtheorem{theorem}{Theorem}[section]

\newtheorem{corollary}[theorem]{Corollary}

\theoremstyle{definition}

\newtheorem{example}{Example}[section]

\titleformat{\section}
{\normalfont\normalsize\bf}{\thesection.}{6pt}{}
\titleformat{\subsection}
{\normalfont\normalsize\bf}{\thesubsection.}{6pt}{}
\titleformat{\subsubsection}
{\normalfont\normalsize}{\thesubsubsection.}{6pt}{}

\newcommand{\bibb}[1]{\left\{#1\right\}}

\newcommand{\smbb}[1]{\left(#1\right)}

\newcommand{\la}{\lambda}
\newcommand{\si}{\sigma}

\newcommand{\ze}{\zeta}

\newcommand{\uud}{\,\mathrm{d}}

\newcommand{\Li}{\mathrm{Li}}
\newcommand{\Lii}{\,\mathrm{Li}}

\newcommand{\ol}{\overline}

\makeatletter

\newdimen\bibspace
\renewenvironment{thebibliography}[1]{%
 \section*{\refname 
       \@mkboth{\MakeUppercase\refname}{\MakeUppercase\refname}}%
     \list{\@biblabel{\@arabic\c@enumiv}}%
          {\settowidth\labelwidth{\@biblabel{#1}}%
           \leftmargin\labelwidth
           \advance\leftmargin\labelsep
           \itemsep\bibspace
           \parsep\z@skip     %
           \@openbib@code
           \usecounter{enumiv}%
           \let\p@enumiv\@empty
           \renewcommand\theenumiv{\@arabic\c@enumiv}}%
     \sloppy\clubpenalty4000\widowpenalty4000%
     \sfcode`\.\@m}
    {\def\@noitemerr
      {\@latex@warning{Empty `thebibliography' environment}}%
     \endlist}

\makeatother

\begin{document}

\title{\bf\boldmath{Explicit formulas of Euler sums via multiple zeta values}}
\author
{
Ce Xu\thanks{\emph{E-mail address\,:}
15959259051@163.com (Ce Xu).}
\quad
Weiping Wang\thanks{Corresponding author. \emph{E-mail addresses\,:}
wpingwang@yahoo.com, wpingwang@zstu.edu.cn (Weiping Wang).}
\\
\small $\ast$ School of Mathematical Sciences, Xiamen University,
    Xiamen 361005, PR China\\
\small $\dag$ School of Science, Zhejiang Sci-Tech University,
    Hangzhou 310018, PR China
}
\date{}
\maketitle

\vspace{-0.5cm}
\begin{center}
\parbox{6.3in}{\small{\bf Abstract}\vspace{3pt}

\hspace{3.5ex}Flajolet and Salvy pointed out that every Euler sum is a $\mathbb{Q}$-linear combination of multiple zeta values. However, in the literature, there is no formula completely revealing this relation. In this paper, using permutations and compositions, we establish two explicit formulas for the Euler sums, and show that all the Euler sums are indeed expressible in terms of MZVs. Moreover, we apply this method to the alternating Euler sums, and show that all the alternating Euler sums are reducible to alternating MZVs. Some famous results, such as the Euler theorem, the Borwein--Borwein--Girgensohn theorems, and the Flajolet--Salvy theorems can be obtained directly from our theory. Some other special cases, such as the explicit expressions of $S_{r^m,q}$, $S_{r^m,\bar{q}}$, $S_{\bar{r}^m,q}$ and $S_{\bar{r}^m,\bar{q}}$, are also presented here. The corresponding Maple programs are developed to help us compute all the sums of weight $w\leq 11$ for non-alternating case and of weight $w\leq 6$ for alternating case.
}

\vspace{6pt}
\parbox{6.3in}{\small{\emph{AMS classification\,:}}\,\,
40A25; 11B83; 11M32}

\vspace{0.5pt}
\parbox{6.3in}{\small{\emph{Keywords\,:}}\,\,
Compositions; Euler sums; Harmonic numbers; Multiple zeta values; Permutations}
\end{center}




\section{Introduction}\label{Sec.intro}

The \emph{generalized harmonic numbers} are defined by
\[
H_0^{(r)}=0\quad\text{and}\quad H_n^{(r)}=\sum_{k=1}^n\frac{1}{k^r}
    \quad\text{for }n,r=1,2,\ldots.
\]
When $r=1$, they reduce to the classical harmonic numbers, denoted as $H_n=H_n^{(1)}$.

Let $\pi=(\pi_1,\pi_2,\ldots,\pi_k)$ be a partition of the positive integer $p$ into $k$ summands, that is, $p=\pi_1+\pi_2+\cdots+\pi_k$ and $\pi_1\leq\pi_2\leq\cdots\leq\pi_k$. Let $q$ be a positive integer with $q\geq 2$. Define the \emph{Euler sum} of index $(\pi,q)$ by
\begin{equation}\label{ES}
S_{\pi,q}:=S_{\pi_1\pi_2\cdots\pi_k,q}
    =\sum_{n=1}^{\infty}\frac{H_n^{(\pi_1)}H_n^{(\pi_2)}\cdots H_n^{(\pi_k)}}{n^q}\,.
\end{equation}
The quantity $\pi_1+\cdots+\pi_k+q$ is called the ``\emph{weight}'' of the sum, and the quantity $k$ is called the ``\emph{degree}'' (see \cite[Section 1]{FlSa98}). Since repeated summands in partitions are indicated by powers, we denote, for instance, the sum
\[
S_{1^22^35,q}:=S_{112225,q}
    =\sum_{n=1}^{\infty}\frac{H_n^2(H_n^{(2)})^{3}H_n^{(5)}}{n^q}\,.
\]
The Euler sums have a history that can be traced back to 1742. In response to a letter from Goldbach, Euler considered sums $S_{p,q}$ (see Berndt \cite[p. 253]{Berndt85.1}). These sums are named as \emph{linear Euler sums} today, and those involving products of at least two harmonic numbers are called \emph{nonlinear Euler sums}.

Additionally, let
\[
\bar{H}_0^{(r)}=0\quad\text{and}\quad \bar{H}_n^{(r)}=\sum_{k=1}^n\frac{(-1)^{k-1}}{k^r}
    \quad\text{for }n,r=1,2,\ldots
\]
be the \emph{alternating harmonic numbers}, and define the two kinds of \emph{alternating Euler sums} by
\begin{align}
&S_{i_1\cdots i_l\bar{i}_{l+1}\cdots \bar{i}_m,q}=\sum_{n=1}^{\infty}
    \frac{H_n^{(i_1)}\cdots H_n^{(i_l)}\bar{H}_n^{(i_{l+1})}\cdots\bar{H}_n^{(i_m)}}{n^q}
    \,,\label{AES1}\\
&S_{i_1\cdots i_l\bar{i}_{l+1}\cdots \bar{i}_m,\bar{q}}=\sum_{n=1}^{\infty}(-1)^{n-1}
    \frac{H_n^{(i_1)}\cdots H_n^{(i_l)}\bar{H}_n^{(i_{l+1})}\cdots\bar{H}_n^{(i_m)}}{n^q}
    \,,\label{AES2}
\end{align}
where $1\leq i_1\leq\cdots\leq i_l$ and $1\leq i_{l+1}\leq\cdots\leq i_m$. It is obvious that when $m-l=0$, the alternating Euler sums (\ref{AES1}) reduce to the classical non-alternating Euler sums (\ref{ES}).

Besides Euler and Goldbach, many mathematicians studied the problem of Euler sums. For example, in 1994, Bailey et al. \cite{BaiBG94.EM} conducted a rather extensive numerical search to determine whether or not some particular (alternating) Euler sums can be expressed as rational linear combinations of several given constants. In 1995, Borwein et al. \cite{BorBG95.PEM} proved that the quadratic sums $S_{1^2,q}$ can reduce to linear sums $S_{2,q}$ and polynomials in zeta values, i.e., the values of the Riemann zeta function $\ze(s)=\sum_{k=1}^{\infty}1/k^s$ at the positive integers. In 1998, Flajolet and Salvy \cite{FlSa98} used the contour integral representations and residue computation to study several classes of Euler sums. Other works on (alternating) Euler sums can be found in \cite{Adam97,BorBor95.OAI,Boyad02,ChenKW16,ChenChu09,ChSri05,Chu97.HSR,Coff03,CopCa10.AZF,DeDo91,
Frei05,Fur11,Mezo14,PanPro05,RaSri02,ShenLC95,Sit87,SunPing03,SunPing07,Val16.1,ValFur16.2,
ZhengDY07}. Despite the tremendous efforts, it can be seen that most of the (alternating) sums remain unknown, and how to compute these sums by general formulas is still a challenge to us.

Recently, rapid progress has been made in this field. Using the Bell polynomials, generating functions, integrals of special functions, multiple zeta (star) values, the Stirling sums and the Tornheim type series, we study the (alternating) Euler sums systematically \cite{WangLyu18.ESSS,WangXu17.ESW,Xu18.EEHS,Xu17.MZVES,Xu17.SECES,XuCai18.OHNN,XuCh16.FACM,
XuCh16.7Sum,XuLi17.TTS,XuYS16.ESI,XuYZ17.EEQES}. As a consequence, the evaluation of all the unknown Euler sums up to the weight $11$ are presented, and a basis of Euler sums of weight $3\leq w \leq 11$ is
\[
\begin{array}{ll}
w=3:&\ze(3),\\
w=4:&\ze(4),\\
w=5:&\ze(5),\ \ze(2)\ze(3),\\
w=6:&\ze(6),\ \ze^2(3),\\
w=7:&\ze(7),\ \ze(2)\ze(5),\ \ze(3)\ze(4),\\
w=8:&\ze(8),\ \ze(3)\ze(5),\ \ze(2)\ze^2(3),\ S_{2,6},\\
w=9:&\ze(9),\ \ze(2)\ze(7),\ \ze(3)\ze(6),\ \ze(4)\ze(5),\ \ze^3(3),\\
w=10:&\ze(10),\ \ze(3)\ze(7),\ \ze^2(5),\ \ze(2)\ze(3)\ze(5),\ \ze^2(3)\ze(4),
    \ \ze(2)S_{2,6},\ S_{2,8},\\
w=11:&\ze(11),\ \ze(2)\ze(9),\ \ze(3)\ze(8),\ \ze(4)\ze(7),\ \ze(5)\ze(6),\ \ze^2(3)\ze(5),
    \ \ze(2)\ze^3(3),\\
    &\ze(3)S_{2,6},\ S_{12,8}.
\end{array}
\]

On the other hand, according to Hoffman \cite{Hoff92} and Zagier \cite{Zag92}, the \emph{multiple zeta values} (MZVs) are defined by
\[
\ze(s_1,s_2,\ldots,s_k):=\sum_{n_1>n_2>\cdots>n_k\geq1}
    \frac{1}{n_1^{s_1}n_2^{s_2}\cdots n_k^{s_k}}\,,
\]
where $s_1,s_2,\ldots,s_k$ are all positive integers with $s_1>1$. The $k$ in the above definition is named the ``\emph{length}'' or ``\emph{depth}'' of a MZV, and the $n=s_1+s_2+\cdots+s_k$ is known as the ``\emph{weight}''. Usually, when writing MZVs, we denote $n$ repetitions of a substring by $\{\cdots\}_n$. For example, $\ze(2,1,2,1,3)=\ze(\{2,1\}_2,3)$. The partial sums
\begin{equation}\label{MHS}
\ze_n(s_1,s_2,\ldots,s_k):=\sum_{n\geq n_1>n_2>\cdots>n_k\geq 1}
    \frac{1}{n_1^{s_1}n_2^{s_2}\cdots n_k^{s_k}}
\end{equation}
are called the \emph{multiple harmonic sums} (MHSs) \cite{HPHPT14}. By convention, $\ze_n(s_1,s_2,\ldots,s_k)=0$ for $n<k$, and $\ze_n(\emptyset)=1$.

Similarly, the \emph{alternating multiple harmonic sums} (alternating MHSs) are defined by
\[
\ze_n(s_1,s_2,\ldots,s_k;\si_1,\si_2,\ldots,\si_k)=\sum_{n\geq n_1>n_2>\cdots>n_k\geq 1}
    \frac{\si_1^{n_1}\si_2^{n_2}\cdots \si_k^{n_k}}{n_1^{s_1}n_2^{s_2}\cdots n_k^{s_k}}\,,
\]
where $s_j$ are positive integers, $\si_j=\pm1$, for $j=1,2,\ldots,k$, with $(s_1,\si_1)\neq(1,1)$. The limit cases of alternating MHSs give rise to \emph{alternating multiple zeta values} (alternating MZVs for short; also called $k$-fold Euler sums) \cite{BBV2010,BorBB97,HPHPT14}:
\begin{equation}\label{AMZV.def1}
\ze(s_1,s_2,\ldots,s_k;\si_1,\si_2,\ldots,\si_k)
    =\lim_{n\to\infty}\ze_n(s_1,s_2,\ldots,s_k;\si_1,\si_2,\ldots,\si_k)\,.
\end{equation}
For convenience, when writing alternating MHSs and MZVs, we shall combine the strings of exponents and signs into a single string, with $s_j$ in the $j$th position when $\si_j=+1$, and $\bar{s}_j$ in the $j$th position when $\si_j=-1$. For example,
\[
\ze(\bar{s}_1,s_2,\ldots,\bar{s}_k)=\sum_{n_1>n_2>\cdots>n_k\geq 1}
    \frac{(-1)^{n_1+n_k}}{n_1^{s_1}n_2^{s_2}\cdots n_k^{s_k}}\,.
\]
In particular, it is known that
\begin{equation}\label{AZV}
\ze(\bar{s})=\sum_{n=1}^{\infty}\frac{(-1)^n}{n^s}=(2^{1-s}-1)\ze(s)\,,
\end{equation}
with $\ze(\bar{1})=-\ln(2)$.

MZVs and alternating MZVs have been of interest to mathematicians and physicists for a long time.
The systematic study of them began in the early 1990s with the works of Hoffman \cite{Hoff92,Hoff97} and Zagier \cite{Zag92}, but the number $\ze(\bar{6},\bar{2})$ appeared in the quantum field theory literature in 1986 \cite{Br1986}, well before the phrase ``multiple zeta values'' had been coined. They are essential to the connection of knot theory with quantum field theory \cite{Br2013,Kassel95}, and they became even more important when higher order calculations in quantum electrodynamics and quantum chromodynamics started to need the multiple harmonic sums \cite{Bl1999,BBV2010}.

According to Minh and Petitot \cite{MinhPet2000,Peti09}, all the MZVs up to the weight $16$ are reducible to $\ze(2)$, $\ze(6,2)$, $\ze(8,2)$, $\ze(10,2)$, $\ze(8,2,1)$, $\ze(8,2,1,1)$, etc., and $\ze(s)$, where $s=3,5,7,9,11,13,15$. The evaluations of the MZVs up to the weight $9$ and some MZVs of weight $10$ are presented in \cite{MinhPet2000}, and a Maple program to compute MZVs up to the weight $16$ is given in \cite{Peti09}. Bl\"{u}mlein et al. \cite{BBV2010} further investigated the MZVs up to the weight $22$ to derive basis-representations for all individual values by the Multiple Zeta Value data mine (henceforth MZVDM). Similarly, extensive tables of the alternating MZVs were compiled, first by Bigotte et al. by means of Lyndon words and shuffle algebras \cite{BJOP2002}, and later by Bl\"{u}mlein et al. as part of MZVDM \cite{BBV2010}. A basis for the $\mathbb{Q}$-vector space spanned by the set of alternating MZVs of weight $1\leq w\leq 6$ is also given in \cite{BBV2010}:
\[
\begin{array}{ll}
w=1:&\ln(2),\\
w=2:&\ze(2),\ \ln^2(2),\\
w=3:&\ze(3),\ \ze(2)\ln2,\ \ln^3(2),\\
w=4:&\Li_4(1/2),\ \ze(4),\ \ze(3)\ln(2),\ \ze(2)\ln^2(2),\ \ln^4(2),\\
w=5:&\Li_5(1/2),\ \ze(5),\ \Li_4(1/2)\ln(2),\ \ze(4)\ln(2),\ \ze(3)\ze(2),\ \ze(3)\ln^2(2),\\
    &\ze(2)\ln^3(2),\ \ln^5(2),\\
w=6:&\Li_6(1/2),\ \ze(6),\ \ze(\bar{5},\bar{1}),\ \Li_5(1/2)\ln(2),\ \ze(5)\ln(2),\
        \Li_4(1/2)\ze(2),\ \Li_4(1/2)\ln^2(2),\\
    &\ze(4)\ln^2(2),\ \ze^2(3),\ \ze(3)\ze(2)\ln(2),\ \ze(3)\ln^3(2),\ \ze(2)\ln^4(2),\
        \ln^6(2),
\end{array}
\]
where $\Li_q(t)$ are the polylogarithm functions defined by
$\Li_q(t)=\sum_{k=1}^{\infty}t^k/k^q$ for $|t|<1$.

As remarked by Flajolet and Salvy \cite{FlSa98}, every Euler sum of weight $w$ and degree $k$ is a $\mathbb{Q}$-linear combination of MZVs of weight $w$ and depth at most $k+1$. However, in the literature, there is no formula completely revealing this relation. In this paper, using the properties of permutations and compositions, we establish two explicit formulas for the Euler sums, and show that all the Euler sums are indeed expressible in terms of MZVs. Moreover, we apply this method to the alternating Euler sums, and show that all the alternating Euler sums are reducible to alternating MZVs. The corresponding Maple programs are also provided, so that the expressions of the (alternating) Euler sums in terms of (alternating) MZVs can be computed automatically. If we further use the evaluations of (alternating) MZVs, the explicit formulas of (alternating) Euler sums in terms of the bases can be finally determined.

The paper is organized as follows. In Sections \ref{Sec.formulas} and \ref{Sec.ge.re}, we use the permutations and compositions to establish the explicit formulas of (alternating) Euler sums in terms of (alternating) MZVs. In Sections \ref{Sec.sp.ca} and \ref{Sec.AES.spca}, we consider some special cases of the general results, including the expressions of $S_{r^m,q}$, $S_{r^m,\bar{q}}$, $S_{\bar{r}^m,q}$ and $S_{\bar{r}^m,\bar{q}}$. Several famous results can also be obtained, such as the Euler theorem, the Borwein--Borwein--Girgensohn theorems and the Flajolet--Salvy theorems. Finally, in Section \ref{Sec.Maple}, we introduce our Maple package briefly.


\section{Explicit formulas of Euler sums in terms of MZVs}\label{Sec.formulas}

By convention, let $\mathcal{S}_m$ be the \emph{symmetric group} of all the permutations on $m$ symbols. In addition, define a \emph{composition} of $m$ be an ordered tuple of positive integers whose elements sum to $m$, and let $\mathcal{C}_m$ be the set of all the compositions of $m$. For example, there are eight compositions of 4:
\[
\mathcal{C}_4=\{(4),(1,3),(2,2),(3,1),(1,1,2),(1,2,1),(2,1,1),(1,1,1,1)\}\,.
\]
Using permutations and compositions, the following explicit formula of the Euler sums can be established.

\begin{theorem}\label{Th.ES.eq1}
Let $i_1,i_2,\ldots,i_m\geq 1$ and $q\geq 2$ be positive integers. For a composition $\xi:=(\xi_1,\xi_2,\ldots,\xi_p)\in\mathcal{C}_m$ and a permutation $\si\in\mathcal{S}_m$, denote
\[
J_c(\xi,\si)=i_{\si(\xi_1+\cdots+\xi_{c-1}+1)}+\cdots+i_{\si(\xi_1+\cdots+\xi_c)}\,,
    \quad\text{for } c=1,2,\ldots,p\,.
\]
Then the Euler sums $S_{i_1i_2\cdots i_m,q}$ are expressible in terms of MZVs:
\begin{equation}\label{ES.eq1}
S_{i_1i_2\cdots i_m,q}
    =\sum_{\xi\in\mathcal{C}_m}\sum_{\si\in\mathcal{S}_m}
        \frac{1}{\xi_1!\xi_2!\cdots\xi_p!}
        \bibb{
            \begin{array}{c}
                \ze(q,J_1(\xi,\si),J_2(\xi,\si),\ldots,J_p(\xi,\si))\\
                +\ze(q+J_1(\xi,\si),J_2(\xi,\si),\ldots,J_p(\xi,\si))
            \end{array}
        }\,.
\end{equation}
\end{theorem}

\begin{proof}
Let us compute
\begin{equation}\label{pro.HarNu.2}
\prod_{j=1}^mH_n^{(i_j)}=\sum_{n\geq n_1,n_2,\ldots,n_m\geq 1}\frac{1}{n_1^{i_1}n_2^{i_2}\cdots n_m^{i_m}}\,.
\end{equation}
To deal with the multiple sum on the right, we should consider all the possible weak orderings on $\{n_1,n_2,\ldots,n_m\}$. These fall into natural classes related to compositions of $m$. Given a composition $\xi=(\xi_1,\xi_2,\ldots,\xi_p)\in\mathcal{C}_m$, we obtain the natural ordering
\begin{align}
n\geq n_1=\cdots=n_{\xi_1}
    &>n_{\xi_1+1}=\cdots=n_{\xi_1+\xi_2}\nonumber\\
    &>\cdots>n_{\xi_1+\cdots+\xi_{p-1}+1}=\cdots=n_m\geq 1\,,\label{xi.ordering.1}
\end{align}
which corresponds to the multiple harmonic sum
\[
\ze_n(i_1+\cdots+i_{\xi_1},i_{\xi_1+1}+\cdots+i_{\xi_1+\xi_2},\ldots,
    i_{\xi_1+\cdots+\xi_{p-1}+1}+\cdots+i_m)\,.
\]
It is evident that any ordering in the class of composition $\xi=(\xi_1,\xi_2,\ldots,\xi_p)$ can be obtained by applying a permutation of $\{1,2,\ldots,m\}$, say $\si$, to (\ref{xi.ordering.1}):
\begin{align*}
n\geq n_{\si(1)}=\cdots=n_{\si(\xi_1)}&>n_{\si(\xi_1+1)}=\cdots=n_{\si(\xi_1+\xi_2)}\\
    &>\cdots>n_{\si(\xi_1+\cdots+\xi_{p-1}+1)}=\cdots=n_{\si(m)}\geq 1\,.
\end{align*}
But swaps of equal elements do not matter, so we must divide the intermediate result by $\xi_1!\xi_2!\cdots\xi_p!$. Thus, the distinct weak orderings in the class $\xi=(\xi_1,\xi_2,\cdots,\xi_p)$ contribute
\[
\sum_{\si\in\mathcal{S}_m}\frac{\ze_n(J_1(\xi,\si),J_2(\xi,\si),\ldots,J_p(\xi,\si))}
    {\xi_1!\xi_2!\cdots\xi_p!}
\]
to the sum (\ref{pro.HarNu.2}), and summing over $\mathcal{C}_m$ gives
\begin{equation}\label{pro.HarNu.1}
\prod_{j=1}^mH_n^{(i_j)}
    =\sum_{\xi\in\mathcal{C}_m}\sum_{\si\in\mathcal{S}_m}
    \frac{\ze_n(J_1(\xi,\si),J_2(\xi,\si),\ldots,J_p(\xi,\si))}
    {\xi_1!\xi_2!\cdots\xi_p!}\,.
\end{equation}
Finally, multiplying (\ref{pro.HarNu.1}) by $1/n^q$, summing over $n$, and using the definitions of Euler sums and MZVs, we obtain the result (\ref{ES.eq1}).
\end{proof}


Besides Theorem \ref{Th.ES.eq1}, we can establish another explicit formula of the Euler sums by permutations and compositions.

\begin{theorem}\label{Th.ES.eq2}
Let $i_1,i_2,\ldots,i_m\geq 2$ and $q\geq 2$ be positive integers. For a composition $\eta:=(\eta_1,\eta_2,\ldots,\eta_p)\in\mathcal{C}_l$ and a permutation $\tau\in\mathcal{S}_l$, denote
\[
\tilde{J}_c(\eta,\tau)
    =i_{j_{\tau(\eta_1+\cdots+\eta_{c-1}+1)}}+\cdots+i_{j_{\tau(\eta_1+\cdots+\eta_c)}}\,,
    \quad\text{for } c=1,2,\ldots,p\,.
\]
Then the Euler sums $S_{i_1i_2\cdots i_m,q}$ are expressible in terms of MZVs:
\begin{align}
&S_{i_1i_2\cdots i_m,q}\nonumber\\
&\quad=\sum_{l=0}^m(-1)^l
    \sum_{1\leq j_1<\cdots<j_l\leq m}
        \frac{\prod\limits_{k=1}^m\ze(i_k)}{\prod\limits_{k=1}^l\ze(i_{j_k})}
        \sum_{\eta\in\mathcal{C}_l}\sum_{\tau\in\mathcal{S}_l}
        \frac{\ze(\tilde{J}_1(\eta,\tau),\tilde{J}_2(\eta,\tau),\ldots,
            \tilde{J}_p(\eta,\tau),q)}
            {\eta_1!\eta_2!\cdots\eta_p!}\,.\label{ES.eq2}
\end{align}
\end{theorem}

\begin{proof}
The product of harmonic numbers can be rewritten as
\begin{align*}
\prod_{j=1}^mH_n^{(i_j)}&=\prod_{j=1}^m\{\ze(i_j)-(\ze(i_j)-H_n^{(i_j)})\}\nonumber\\
    &=\sum_{l=0}^m(-1)^l\sum_{1\leq j_1<\cdots<j_l\leq m}
        \frac{\prod\limits_{k=1}^m\ze(i_k)}{\prod\limits_{k=1}^l\ze(i_{j_k})}
        \prod_{k=1}^l\{\ze(i_{j_k})-H_n^{(i_{j_k})}\}\,.
\end{align*}
Therefore, the Euler sums $S_{i_1i_2\cdots i_m,q}$ satisfy
\begin{equation}\label{ES.eq2.temp}
S_{i_1i_2\cdots i_m,q}=\sum_{l=0}^m(-1)^l\sum_{1\leq j_1<\cdots<j_l\leq m}
    \frac{\prod\limits_{k=1}^m\ze(i_k)}{\prod\limits_{k=1}^l\ze(i_{j_k})}
    \sum_{n=1}^{\infty}\frac{1}{n^q}\prod_{k=1}^l\{\ze(i_{j_k})-H_n^{(i_{j_k})}\}\,.
\end{equation}
Similarly to the proof of (\ref{pro.HarNu.1}), by means of permutations and compositions, we have
\begin{align*}
\prod_{k=1}^l\{\ze(i_{j_k})-H_n^{(i_{j_k})}\}
    &=\sum_{n_1,n_2,\ldots,n_l>n}\frac{1}{n_1^{i_{j_1}}n_2^{i_{j_2}}\cdots n_l^{i_{j_l}}}\\
    &=\sum_{\eta\in\mathcal{C}_l}\sum_{\tau\in\mathcal{S}_l}
        \frac{\ze_{>n}(\tilde{J}_1(\eta,\tau),\tilde{J}_2(\eta,\tau),\ldots,\tilde{J}_p(\eta,\tau))}
        {\eta_1!\eta_2!\cdots\eta_p!}\,,
\end{align*}
where $\ze_{>n}(s_1,s_2,\ldots,s_k)$ are multiple sums defined by
\[
\ze_{>n}(s_1,s_2,\ldots,s_k):=\sum_{n_1>n_2>\cdots>n_k>n}\frac{1}{n_1^{s_1}n_2^{s_2}\cdots n_k^{s_k}}\,.
\]
Then using the definition of MZVs, we obtain from (\ref{ES.eq2.temp}) the final result (\ref{ES.eq2}).
\end{proof}


\section{Some special cases of Theorems \ref{Th.ES.eq1} and \ref{Th.ES.eq2}}\label{Sec.sp.ca}

In this section, by the evaluations and identities of MZVs, we give some special cases of Theorems \ref{Th.ES.eq1} and \ref{Th.ES.eq2}. Some famous results will be obtained, such as the Euler theorem, the Borwein--Borwein--Girgensohn theorems and the Flajolet--Salvy theorems.

According to \cite[Section 6]{Adam97}, the MZVs $\ze(k+1,\{1\}_l)$ can be determined by the values of the integrals $W(k,l)$:\begin{equation}\label{ze.k.1}
\ze(k+1,\{1\}_l)=\frac{(-1)^{k+l}}{k!l!}W(k,l)
\end{equation}
(see also \cite[Eq. (2.18)]{WangXu17.ESW} and \cite[Eq. (2.27)]{Xu17.MZVES}). Here the integrals
\[
W(k,l):=\int_0^1\frac{\ln^k(t)\ln^l(1-t)}{1-t}\uud t
\]
were introduced by K\"{o}lbig \cite{Kolbig82.CEF,Kolbig86.NGP} in 1980s, and reconsidered recently by us \cite{WangLyu18.ESSS,WangXu17.ESW,Xu17.MZVES,XuYS16.ESI} to study the Euler sums. From these works, we know that
\begin{align}
W(k,l)&=\frac{(-1)^{k+l}(k+l)!\ze(k+l+1)}{l+1}\nonumber\\
&\quad-\sum_{i=1}^{k-1}\sum_{j=1}^{l}\binom{k-1}{i-1}\binom{l}{j}
    (-1)^{i+j}(i+j-1)!\ze(i+j)W(k-i,l-j)\label{Wkl.rec}
\end{align}
for $k\geq 1$ and $l\geq 0$, with
\[
W(1,l)=(-1)^{l+1}l!\ze(l+2)\,,\quad
W(k,0)=(-1)^kk!\ze(k+1)\,,
\]
and
\begin{equation}\label{Wkl.sym}
\frac{W(k,l-1)}{k!(l-1)!}=\frac{W(l,k-1)}{l!(k-1)!}\,.
\end{equation}
By (\ref{ze.k.1})--(\ref{Wkl.sym}), the MZVs $\ze(k+1,\{1\}_l)$ are reducible to zeta values. For example, we have
\begin{align*}
\ze(q,1)&=\frac{q}{2}\ze(q+1)-\frac{1}{2}\sum_{i=1}^{q-2}\ze(i+1)\ze(q-i)\,,\\
\ze(q,1,1)&=\frac{q(q+1)}{6}\ze(q+2)+\frac{1}{2}\ze(2)\ze(q)
        -\frac{q}{4}\sum_{j=0}^{q-2}\ze(j+2)\ze(q-j)\nonumber\\
    &\quad+\frac{1}{6}\sum_{j=2}^{q-2}\ze(q-j)\sum_{i=0}^{j-2}\ze(i+2)\ze(j-i)\,,
\end{align*}
which, combined with Theorem \ref{Th.ES.eq1}, give the next two corollaries.

\begin{corollary}[the Euler theorem and the Borwein--Borwein--Girgensohn theorem]
For integers $p\geq 1$ and $q\geq 2$, the linear Euler sums $S_{p,q}$ satisfy
\begin{align}
&S_{p,q}=\ze(q,p)+\ze(p+q)\,,\label{S.pq}\\
&S_{1,q}=\smbb{\frac{q}{2}+1}\ze(q+1)-\frac{1}{2}\sum_{i=1}^{q-2}\ze(i+1)\ze(q-i)\,.\label{S.1q}
\end{align}
In particular, for any odd weight $w=p+q$, the sums $S_{p,q}$ satisfy
\begin{align}
S_{p,q}&=\frac{1-(-1)^p}{2}\ze(p)\ze(q)
    +\frac{1}{2}\ze(w)\bibb{1-(-1)^p\binom{w-1}{q}-(-1)^p\binom{w-1}{p}}\nonumber\\
&\quad+(-1)^p\sum_{k=1}^{\frac{w-1}{2}}
    \bibb{\binom{w-2k-1}{p-1}+\binom{w-2k-1}{q-1}}\ze(2k)\ze(w-2k)\,,\label{S.pq.odd}
\end{align}
where $\ze(1)$ should be interpreted as $0$ wherever it occurs.
\end{corollary}

\begin{proof}
Setting $m=1$ and $i_1=p$ in Theorem \ref{Th.ES.eq1}, we obtain (\ref{S.pq}), which can also be verified directly by the definitions of Euler sums and MZVs (see, e.g., \cite[Eq. (2.23)]{WangXu17.ESW}). Setting $p=1$ in (\ref{S.pq}) yields (\ref{S.1q}), which is in fact the famous Euler theorem (see, e.g., \cite[Theorem 2.2]{FlSa98}). According to Borwein et al. \cite[Section 4]{BorBB97} and Teo \cite[Theorem 2.10]{Teo18}, the (alternating) MZVs of depth 2 satisfy
\begin{equation}\label{MZV.st}
\ze(s,t;\si,\tau)=\frac{1}{2}\bibb{-\la_{s+t}+(1+(-1)^s)\ze(s;\si)\ze(t;\tau)+\mu_{s+t}}
    -\sum_{0<k<(s+t)/2}\la_{2k}\mu_{s+t-2k}\,,
\end{equation}
where $s+t$ is odd, $\la_r=\ze(r;\si\tau)$ and
\[
\mu_r=(-1)^s\bibb{\binom{r-1}{s-1}\ze(r;\si)+\binom{r-1}{t-1}\ze(r;\tau)}\,.
\]
Thus, when $p+q$ is odd and $p,q>1$, by the Euler reflection formula
\begin{equation}\label{z.ab}
\ze(a,b)+\ze(b,a)=\ze(a)\ze(b)-\ze(a+b)
\end{equation}
and (\ref{MZV.st}), we obtain (\ref{S.pq.odd}), which also holds for $p=1$ if we denote $\ze(1):=0$. Note that (\ref{S.pq.odd}) is one of the Borwein--Borwein--Girgensohn theorems \cite[Theorem 3.1]{FlSa98}, extrapolated first by Euler without proof, and verified by Borwein et al. in \cite{BorBG95.PEM}.
\end{proof}

\begin{corollary}[the Flajolet--Salvy theorem and the Borwein--Borwein--Girgensohn theorem]
For integers $i_1,i_2\geq 1$ and $q\geq 2$, the quadratic Euler sums $S_{i_1i_2,q}$ satisfy
\begin{align}
S_{i_1i_2,q}&=\ze(q,i_1+i_2)+\ze(q,i_1,i_2)+\ze(q,i_2,i_1)\nonumber\\
    &\quad+\ze(q+i_1+i_2)+\ze(q+i_1,i_2)+\ze(q+i_2,i_1)\,.\label{S.r1r2q}
\end{align}
Thus, for any even weight $w=i_1+i_2+q$, with $i_1,i_2>1$, the sums $S_{i_1i_2,q}$ are reducible to linear sums and zeta values. Moreover, the sums $S_{1^2,q}$ satisfy
\begin{align}
S_{1^2,q}&=S_{2,q}+\frac{(q+1)(q+3)}{3}\ze(q+2)+\ze(2)\ze(q)
    -\frac{q+2}{2}\sum_{j=0}^{q-2}\ze(j+2)\ze(q-j)\nonumber\\
    &\quad+\frac{1}{3}\sum_{j=2}^{q-2}\ze(q-j)\sum_{i=0}^{j-2}\ze(i+2)\ze(j-i)\label{S.11q}\,.
\end{align}
Thus, the sums $S_{1^2,q}$ reduce to linear sums $S_{2,q}$ and polynomials in zeta values, and for any odd integer $q\geq 3$, the sums $S_{1^2,q}$ reduce to zeta values.
\end{corollary}

\begin{proof}
Set $m=2$ in Theorem \ref{Th.ES.eq1}. Since $\mathcal{C}_2=\{(2),(1,1)\}$ and $\mathcal{S}_2=\{(1)(2),(1,2)\}$, we obtain (\ref{S.r1r2q}). According to Borwein and Girgensohn \cite{BoGir96}, the MZVs $\ze(q,i_1,i_2)$ and $\ze(q,i_2,i_1)$ are reducible to zeta values and MZVs of depth $2$ when the weight $w=i_1+i_2+q$ is even, then for $i_1,i_2>1$, the sums $S_{i_1i_2,q}$ are reducible to linear sums and zeta values. This is in fact one of the Flajolet--Salvy theorems \cite[Theorem 4.2]{FlSa98}.

Now, setting further $i_1=i_2=1$ in (\ref{S.r1r2q}), using (\ref{S.pq}), and substituting the evaluations of $\ze(q,1)$ and $\ze(q,1,1)$, we obtain (\ref{S.11q}), which is another Borwein--Borwein--Girgensohn theorem \cite[Theorem 4.1]{FlSa98}.
\end{proof}

Similarly, setting $m=3$ in Theorem \ref{Th.ES.eq1} yields the explicit formula of $S_{i_1i_2i_3,q}$, which has 26 terms. Setting further $i_1=i_2=i_3=1$ gives
\begin{align}
S_{1^3,q}&=\ze(q,3)+\ze(q+3)+3\{\ze(q,1,2)+\ze(q+1,2)+\ze(q,2,1)+\ze(q+2,1)\}\nonumber\\
    &\quad+6\{\ze(q,\{1\}_3)+\ze(q+1,1,1)\}\,.\label{S.111q}
\end{align}
Using Eq. (\ref{S.pq}) and
\begin{equation}\label{mzv.eq32}
\ze(k,i,j)+\ze(k,j,i)=S_{ij,k}-S_{i,j+k}-S_{j,i+k}-S_{i+j,k}+2\ze(i+j+k)\,,
\end{equation}
for integers $j\geq i\geq 1$ and $k\geq 2$ (see \cite[Theorem 2.7]{WangXu17.ESW}), we rewrite (\ref{S.111q}) as
\[
S_{1^3,q}=3S_{12,q}-2S_{3,q}+6\ze(q,\{1\}_3)+6\ze(q+1,1,1)\,.
\]
This leads us to the following known result \cite[Theorem 5.1 (1)]{FlSa98}:

\begin{corollary}[the Flajolet--Salvy theorem]
For odd weights, the cubic combinations $S_{1^3,q}-3S_{12,q}$ are expressible in terms of zeta values.
\end{corollary}

In fact, a general formula for the sums $S_{r^m,q}$ can be derived from Theorem \ref{Th.ES.eq1}.

\begin{theorem}\label{Th.srmq1}
For integers $r\geq 1$ and $q\geq 2$, the Euler sums $S_{r^m,q}$ satisfy
\begin{equation}\label{ES.rmq1}
S_{r^m,q}=\sum_{\xi\in\mathcal{C}_m}\binom{m}{\xi_1,\xi_2,\ldots,\xi_p}
    \{\ze(q,r\xi_1,r\xi_2,\ldots,r\xi_p)+\ze(q+r\xi_1,r\xi_2,\ldots,r\xi_p)\}\,,
\end{equation}
where $\xi:=(\xi_1,\xi_2,\ldots,\xi_p)$ and
\[
\binom{m}{\xi_1,\xi_2,\ldots,\xi_p}:=\frac{m!}{\xi_1!\xi_2!\cdots\xi_p!}
\]
are the multinomial coefficients.
\end{theorem}

\begin{proof}
In this case, $i_1=i_2=\cdots=i_m=r$. Thus, given a composition $\xi=(\xi_1,\xi_2,\ldots,\xi_p)\in\mathcal{C}_m$, for any $\si\in\mathcal{S}_m$, there holds $J_c(\xi,\si)=r\xi_c$, for $c=1,2,\ldots,p$, and the corresponding summand equals
\[
\frac{1}{\xi_1!\xi_2!\cdots\xi_p!}
    \{\ze(q,r\xi_1,r\xi_2,\ldots,r\xi_p)+\ze(q+r\xi_1,r\xi_2,\ldots,r\xi_p)\}\,.
\]
Then we obtain the desired result.
\end{proof}

Now, let us establish the explicit formulas for the quadratic sums $S_{r^2,r}$ and the cubic sums $S_{r^3,r}$ from Theorem \ref{Th.srmq1} and identities on MZVs.

By the shuffle relations, the Euler reflection formula (\ref{z.ab}) can be generalized:
\begin{align}
&\ze(a,b,c)+\ze(a,c,b)+\ze(b,a,c)+\ze(b,c,a)+\ze(c,a,b)+\ze(c,b,a)\nonumber\\
&\quad=\ze(a)\ze(b)\ze(c)+2\ze(a+b+c)-\ze(a)\ze(b+c)-\ze(b)\ze(a+c)-\ze(c)\ze(a+b)\,,\label{z.abc}
\end{align}
for $a,b,c\geq2$. Moreover, (\ref{z.ab}) gives
\begin{equation}\label{z.rr}
\ze(r,r)=\frac{1}{2}\{\ze^2(r)-\ze(2r)\}\,,
\end{equation}
which, combined with the recurrence
\[
\ze(\{r\}_m)=\frac{(-1)^{m-1}}{m}\sum_{i=0}^{m-1}(-1)^i\ze(\{r\}_i)\ze(rm-ri)\,,
\]
where $\ze(\{r\}_0):=1$ (see \cite[Eq. (4.6)]{Xu17.MZVES}), yields
\begin{align}
&\ze(\{r\}_3)=\frac{1}{6}\ze^3(r)-\frac{1}{2}\ze(r)\ze(2r)+\frac{1}{3}\ze(3r)\,,\label{z.rrr}\\
&\ze(\{r\}_4)=\frac{1}{24}\ze^4(r)-\frac{1}{4}\ze^2(r)\ze(2r)+\frac{1}{3}\ze(r)\ze(3r)
    +\frac{1}{8}\ze^2(2r)-\frac{1}{4}\ze(4r)\,.\label{z.rrrr}
\end{align}
Using these identities, we obtain the following corollary.

\begin{corollary}\label{Coro.r2r.r3r}
For any integer $r\geq 2$, the quadratic Euler sums $S_{r^2,r}$ and the cubic Euler sums $S_{r^3,r}$ satisfy
\begin{align}
&S_{r^2,r}=\frac{1}{3}\ze^3(r)-\frac{1}{3}\ze(3r)+S_{r,2r}\,,\label{S.rrr}\\
&S_{r^3,r}=\frac{1}{4}\ze^4(r)+\frac{3}{4}\ze^2(2r)+\ze(4r)
    -S_{r,3r}+\frac{3}{2}S_{r^2,2r}-\frac{3}{2}S_{2r,2r}\,.\label{S.rrrr}
\end{align}
Thus, all the sums $S_{r^2,r}$ and $S_{r^3,r}$ are reducible to linear sums and polynomials in zeta values. In particular, for any odd integer $r\geq 3$, the sums $S_{r^2,r}$ are expressible in terms of zeta values.
\end{corollary}

\begin{proof}
From (\ref{ES.rmq1}), we have
\begin{align*}
S_{r^2,r}&=\ze(r,2r)+\ze(3r)+2\ze(r,r,r)+2\ze(2r,r)\,,\\
S_{r^3,r}&=\ze(r,3r)+\ze(4r)+3\{\ze(r,2r,r)+\ze(3r,r)+\ze(r,r,2r)+\ze(2r,2r)\}\\
    &\quad+6\{\ze(\{r\}_4)+\ze(2r,r,r)\}\,.
\end{align*}
By means of (\ref{S.pq}), (\ref{z.ab}), (\ref{mzv.eq32}) and (\ref{z.abc})--(\ref{z.rrrr}), we obtain (\ref{S.rrr}) and (\ref{S.rrrr}). Next, according to the Flajolet-Salvy theorem \cite[Theorem 4.2]{FlSa98}, $S_{r^2,2r}$ are reducible to linear sums and polynomials in zeta values, so the remark in the corollary also holds. Note that (\ref{S.rrr}) has been presented in \cite[Eq. (3.40)]{XuYS16.ESI}.
\end{proof}

\begin{example}
Using Corollary \ref{Coro.r2r.r3r} as well as the evaluations of some linear Euler sums, we have
\begin{align*}
&S_{4^2,4}=\frac{690247}{16584}\ze(12)-16\ze(3)\ze(9)-28\ze(5)\ze(7)+8S_{2,10}\,,\\
&S_{5^2,5}=\frac{4505}{3}\ze(15)-15\ze(8)\ze(7)-70\ze(6)\ze(9)
    -220\ze(4)\ze(11)-715\ze(2)\ze(13)+\frac{1}{3}\ze^3(5)\,,
\end{align*}
and
\begin{align*}
&\begin{aligned}
S_{2^3,2}=\frac{1285}{32}\ze(8)-60\ze(3)\ze(5)+9\ze(2)\ze^2(3)+\frac{31}{2}S_{2,6}\,,
\end{aligned}\\
&\begin{aligned}
S_{3^3,3}&=\frac{478711}{1382}\ze(12)+\frac{9}{2}\ze(3)\ze(9)-309\ze(5)\ze(7)
    +27\ze(2)\ze^2(5)-63\ze(2)\ze(3)\ze(7)\\
    &\quad+\frac{1}{4}\ze^4(3)-\frac{9}{4}S_{2,10}+\frac{63}{2}\ze(2)S_{2,8}\,.
\end{aligned}
\end{align*}
The readers may try to obtain the evaluations of some other Euler sums by this corollary.\hfill\qedsymbol
\end{example}

By specifying the parameters, we can also obtain many special cases from Theorem \ref{Th.ES.eq2}. For example, setting $m=1,2$ gives
\begin{align}
S_{p,q}&=\ze(p)\ze(q)-\ze(p,q)\,,\label{S.pq2}\\
S_{i_1i_2,q}
    &=\ze(i_1)\ze(i_2)\ze(q)-\ze(i_2)\ze(i_1,q)-\ze(i_1)\ze(i_2,q)+\ze(i_1+i_2,q)\nonumber\\
    &\quad+\ze(i_1,i_2,q)+\ze(i_2,i_1,q)\,,\label{S.r1r2q2}
\end{align}
for $i_1,i_2,p,q\geq 2$. Additionally, we have the next result.

\begin{theorem}\label{Th.srmq2}
For integers $r,q\geq 2$, the Euler sums $S_{r^m,q}$ satisfy
\begin{equation}\label{ES.rmq2}
S_{r^m,q}=\sum_{l=0}^m\sum_{\eta\in\mathcal{C}_l}
    (-1)^l\binom{m}{l}\binom{l}{\eta_1,\eta_2,\ldots,\eta_p}
    \ze^{m-l}(r)\ze(r\eta_1,r\eta_2,\ldots,r\eta_p,q)\,,
\end{equation}
where $\eta:=(\eta_1,\eta_2,\ldots,\eta_p)$.
\end{theorem}

Note that equating the two expressions of $S_{p,q}$ in (\ref{S.pq}) and (\ref{S.pq2}) gives the Euler reflection formula (\ref{z.ab}). Equating the two expressions of $S_{i_1i_2,q}$ in (\ref{S.r1r2q}) and (\ref{S.r1r2q2}), expanding the two terms $\ze(i_2)\ze(i_1,q)+\ze(i_1)\ze(i_2,q)$ by the shuffle relation, and applying (\ref{z.ab}) yield the generalized formula (\ref{z.abc}).


\section{Alternating Euler sums and expansion of product of sums}\label{Sec.ge.re}

In this section, we establish a general result on product of sums, and show that the two kinds of alternating Euler sums given in (\ref{AES1}) and (\ref{AES2}) can be evaluated by alternating MZVs.

For positive integers $j$ and $k$, let $a_j(k)$ be functions on $k$. Define the partial sums
\[
S_n(a_j):=\sum_{k=1}^na_j(k)
\]
and
\[
S_n(a_1,a_2,\ldots,a_m)=\sum_{n\geq n_1>n_2>\cdots>n_m\geq 1}a_1(n_1)a_2(n_2)\cdots a_m(n_m)\,.
\]
Now, consider the product of sums $\prod_{j=1}^mS_n(a_j)$. By computation, we have
\[
\prod_{j=1}^mS_n(a_j)=\sum_{n\geq n_1,n_2,\ldots,n_m\geq 1}a_1(n_1)a_2(n_2)\cdots a_m(n_m)\,.
\]
Similarly to the proof of (\ref{pro.HarNu.1}), by discussing all the possible weak orderings on $\{n_1,n_2,\ldots,n_m\}$, and using permutations and compositions, we obtain the following general expansion.

\begin{theorem}\label{Th.gen.exp}
For a composition $\theta:=(\theta_1,\theta_2,\ldots,\theta_p)\in\mathcal{C}_m$ and a permutation $\rho\in\mathcal{S}_m$, denote
\[
T_c(\theta,\rho)
    =a_{\rho(\theta_1+\cdots+\theta_{c-1}+1)}\times\cdots\times a_{\rho(\theta_1+\cdots+\theta_c)}\,,
    \quad\text{for } c=1,2,\ldots,p\,,
\]
be products of functions. Then the product of sums $\prod_{j=1}^mS_n(a_j)$ can be expanded as
\begin{equation}
\prod_{j=1}^mS_n(a_j)
    =\sum_{\theta\in\mathcal{C}_m}\sum_{\rho\in\mathcal{S}_m}
    \frac{S_n(T_1(\theta,\rho),T_2(\theta,\rho),\ldots,T_p(\theta,\rho))}
    {\theta_1!\theta_2!\cdots\theta_p!}\,.
\end{equation}
\end{theorem}

\begin{proof}
Given a composition $\theta:=(\theta_1,\theta_2,\ldots,\theta_p)\in\mathcal{C}_m$ and a permutation $\rho\in\mathcal{S}_m$, we have the natural ordering
\begin{align*}
n\geq n_{\rho(1)}=\cdots=n_{\rho(\theta_1)}
    &>n_{\rho(\theta_1+1)}=\cdots=n_{\rho(\theta_1+\theta_2)}\\
    &>\cdots>n_{\rho(\theta_1+\cdots+\theta_{p-1}+1)}=\cdots=n_{\rho(m)}\geq 1\,,
\end{align*}
which corresponds to the sum
\begin{align*}
\sum_{n\geq n_1>n_2>\cdots>n_p\geq 1}
&\prod_{k=1}^{\theta_1}a_{\rho(k)}(n_1)
    \prod_{k=1}^{\theta_2}a_{\rho(\theta_1+k)}(n_2)\cdots
    \prod_{k=1}^{\theta_p}a_{\rho(\theta_1+\cdots+\theta_{p-1}+k)}(n_p)\\
&=S_n(T_1(\theta,\rho),T_2(\theta,\rho),\ldots,T_p(\theta,\rho))\,.
\end{align*}
Thus, the desired expansion can be established.
\end{proof}

It can be found that Eq. (\ref{pro.HarNu.1}) is a special case of this theorem. Define $a_j(k):=1/k^{i_j}$, for $k\in\mathbb{N}$; then $S_n(a_j)=H_n^{(i_j)}$ and
\[
\prod_{j=1}^mS_n(a_j)=\prod_{j=1}^mH_n^{(i_j)}\,.
\]
By computation, we have
\[
T_c(\theta,\rho):\quad n_c\longmapsto
    \frac{1}{n_c^{i_{\rho(\theta_1+\cdots+\theta_{c-1}+1)}
    +\cdots+i_{\rho(\theta_1+\cdots+\theta_c)}}}\,,
    \quad\text{for } c=1,2,\ldots,p\text{ and }n_c\in\mathbb{N}\,.
\]
Hence, by the definition of MHS, we obtain (\ref{pro.HarNu.1}).

We can also apply Theorem \ref{Th.gen.exp} to products of harmonic numbers and alternating harmonic numbers, and finally establish the explicit formulas of alternating Euler sums (\ref{AES1}) and (\ref{AES2}).

Let $\chi_{\mathbb{N}}$ be the characteristic function of the set of positive integers:
\[
\chi_{\mathbb{N}}(x):=\left\{
        \begin{array}{cc}
            1\,,&x\in\mathbb{N}\,,\\
            0\,,&x\notin\mathbb{N}\,.
        \end{array}
    \right.
\]
Then we have

\begin{theorem}\label{Th.AES1}
The alternating Euler sums $S_{i_1\cdots i_l\bar{i}_{l+1}\cdots\bar{i}_m,q}$ are expressible in terms of alternating MZVs:
\[
S_{i_1\cdots i_l\bar{i}_{l+1}\cdots\bar{i}_m,q}
    =\sum_{\theta\in\mathcal{C}_m}\sum_{\rho\in\mathcal{S}_m}
        \frac{(-1)^{\la_1+\cdots+\la_p}}{\theta_1!\cdots\theta_p!}
        \bibb{
            \begin{array}{c}
                \ze(q,\tilde{T}_1(\theta,\rho),\ldots,\tilde{T}_p(\theta,\rho))\\
                +\ze(\hat{T}_1(\theta,\rho),\tilde{T}_2(\theta,\rho),\ldots,
                    \tilde{T}_p(\theta,\rho))
            \end{array}
        }\,,
\]
where $\theta:=(\theta_1,\theta_2,\ldots,\theta_p)$,
\begin{equation}\label{AES.lac}
\la_c=\sum_{k=1}^{\theta_c}\chi_{\mathbb{N}}(\rho(\theta_1+\cdots+\theta_{c-1}+k)-l)\,,
\end{equation}
and the parameters in the alternating MZVs are defined by
\begin{align}
&\tilde{T}_c(\theta,\rho)=\left\{
        \begin{array}{ll}
            \overline{i_{\rho(\theta_1+\cdots+\theta_{c-1}+1)}
                +\cdots+i_{\rho(\theta_1+\cdots+\theta_c)}}\,,
                &\text{if }\la_c\text{ is odd}\,,\\
            i_{\rho(\theta_1+\cdots+\theta_{c-1}+1)}+\cdots+i_{\rho(\theta_1+\cdots+\theta_c)}\,,
                &\text{if }\la_c\text{ is even}\,,\label{AES.Tc}
        \end{array}
    \right.\\
&\hat{T}_1(\theta,\rho)=\left\{
        \begin{array}{ll}
            \overline{q+i_{\rho(1)}+\cdots+i_{\rho(\theta_1)}}\,,
                &\text{if }\la_1\text{ is odd}\,,\\
            q+i_{\rho(1)}+\cdots+i_{\rho(\theta_1)}\,,
                &\text{if }\la_1\text{ is even}\,,
        \end{array}
    \right.\nonumber
\end{align}
for $c=1,2,\ldots,p$.
\end{theorem}

\begin{proof}
In this case, we should consider the product $\prod_{j=1}^lH_n^{(i_j)}\prod_{j=l+1}^m\bar{H}_n^{(i_j)}$. Since
\[
a_j(k):=\left\{
        \begin{array}{cc}
            1/k^{i_j}\,,&j=1,\ldots,l\,,\\
            (-1)^{k-1}/k^{i_j}\,,&j=l+1,\ldots,m\,,
        \end{array}
    \right.
\]
the functions $T_c(\theta,\rho)=a_{\rho(\theta_1+\cdots+\theta_{c-1}+1)}\times\cdots\times a_{\rho(\theta_1+\cdots+\theta_c)}$ in Theorem \ref{Th.gen.exp} are
\[
T_c(\theta,\rho):\quad n_c\longmapsto\frac{(-1)^{\la_c(n_c-1)}}
    {n_c^{i_{\rho(\theta_1+\cdots+\theta_{c-1}+1)}+\cdots+i_{\rho(\theta_1+\cdots+\theta_c)}}}
    \,,
\]
for $n_c\in\mathbb{N}$ and $c=1,2,\ldots,p$. Therefore, by Theorem \ref{Th.gen.exp}, this product can be expanded as a multiple summation:
\begin{align}
&\prod_{j=1}^lH_n^{(i_j)}\prod_{j=l+1}^m\bar{H}_n^{(i_j)}\nonumber\\
&\quad=\sum_{\theta\in\mathcal{C}_m}\sum_{\rho\in\mathcal{S}_m}
    \frac{1}{\theta_1!\cdots\theta_p!}\sum_{n\geq n_1>\cdots>n_p\geq 1}
    \prod_{c=1}^p\frac{(-1)^{\la_c(n_c-1)}}
    {n_c^{i_{\rho(\theta_1+\cdots+\theta_{c-1}+1)}+\cdots+i_{\rho(\theta_1+\cdots+\theta_c)}}}
    \label{H.pro.exp}\,.
\end{align}
Multiplying the expansion by $1/n^q$ and summing over $n$, we obtain the final result.
\end{proof}

Similarly, using Theorem \ref{Th.gen.exp} to obtain the expansion (\ref{H.pro.exp}), multiplying $(-1)^{n-1}/n^q$ and summing over $n$ yield the explicit formula of $S_{i_1\cdots i_l\bar{i}_{l+1}\cdots\bar{i}_m,\bar{q}}$.

\begin{theorem}\label{Th.AES2}
The alternating Euler sums $S_{i_1\cdots i_l\bar{i}_{l+1}\cdots\bar{i}_m,\bar{q}}$ are expressible in terms of alternating MZVs:
\[
S_{i_1\cdots i_l\bar{i}_{l+1}\cdots\bar{i}_m,\bar{q}}
    =\sum_{\theta\in\mathcal{C}_m}\sum_{\rho\in\mathcal{S}_m}
        \frac{(-1)^{\la_1+\cdots+\la_p+1}}{\theta_1!\cdots\theta_p!}
        \bibb{
            \begin{array}{c}
                \ze(\bar{q},\tilde{T}_1(\theta,\rho),\ldots,\tilde{T}_p(\theta,\rho))\\
                +\ze(\check{T}_1(\theta,\rho),\tilde{T}_2(\theta,\rho),\ldots,
                    \tilde{T}_p(\theta,\rho))
            \end{array}
        }\,,
\]
where $\theta:=(\theta_1,\theta_2,\ldots,\theta_p)$, $\la_c$ and $\tilde{T}_c(\theta,\rho)$ are defined by (\ref{AES.lac}) and (\ref{AES.Tc}), for $c=1,2,\ldots,p$, and the parameter $\check{T}_1(\theta,\rho)$ is defined by
\[
\check{T}_1(\theta,\rho)=\left\{
        \begin{array}{ll}
            \overline{q+i_{\rho(1)}+\cdots+i_{\rho(\theta_1)}}\,,
                &\text{if }\la_1\text{ is even}\,,\\
            q+i_{\rho(1)}+\cdots+i_{\rho(\theta_1)}\,,
                &\text{if }\la_1\text{ is odd}\,.
        \end{array}
    \right.
\]
\end{theorem}


\section{Some special cases of Theorems \ref{Th.AES1} and \ref{Th.AES2}}\label{Sec.AES.spca}

The expansion of non-alternating Euler sums in Theorem \ref{Th.ES.eq1} can be obtained directly from Theorem \ref{Th.AES1} by setting $m=l$. Now, let us consider some other cases of Theorems \ref{Th.AES1} and \ref{Th.AES2}.

Similarly to the original notation of alternating MZVs, to succinctly state the following results, we denote the Euler sums (\ref{ES})--(\ref{AES2}) by
\begin{align*}
&S_{i_1\cdots i_m,q}
    :=S(i_1,\ldots,i_m,q;\{1\}_{m+1})\,,\\
&S_{i_1\cdots i_l\bar{i}_{l+1}\cdots \bar{i}_m,q}
    :=S(i_1,\ldots,i_m,q;\{1\}_l,\{-1\}_{m-l},1)\,,\\
&S_{i_1\cdots i_l\bar{i}_{l+1}\cdots \bar{i}_m,\bar{q}}
    :=S(i_1,\ldots,i_m,q;\{1\}_l,\{-1\}_{m-l},-1)\,,
\end{align*}
respectively. In this way, the Euler sums and alternating Euler sums can be represented in a unified manner. For example, all the linear sums can be denoted by $S(p,q;\si_1,\si_2)$, where the parameter $\si_k=\mp1$, depending on whether or not the $k$th parameter before the semicolon is barred. Then using Theorems \ref{Th.AES1} and \ref{Th.AES2} as well as Eq. (\ref{MZV.st}), we obtain the next corollary. See Flajolet and Salvy's results \cite[Theorems 7.1 and 7.2]{FlSa98}. See also \cite{BaiBG94.EM,Sit87}.

\begin{corollary}[the Flajolet--Salvy theorem]
The (alternating) linear Euler sums $S(p,q;\si_1,\si_2)$ satisfy
\begin{equation}\label{AES.pq}
S(p,q;\si_1,\si_2)=\si_1\si_2\{\ze(q,p;\si_2,\si_1)+\ze(p+q;\si_1\si_2)\}\,.
\end{equation}
In particular, when $p+q$ is odd, the sums $S(p,q;\si_1,\si_2)$ are reducible to (alternating) zeta values.
\end{corollary}

According to \cite[Theorem 4.2]{Xu17.MZVES}, for two sums
\[
A_n(m)=\sum_{k=1}^nx_k^m\quad\text{and}\quad
B_n(m)=\sum_{n\geq k_1>k_2>\cdots>k_m\geq 1}x_{k_1}x_{k_2}\cdots x_{k_m}\,,
\]
where $B_n(0)=1$, we have
\begin{equation}\label{Anm.Bnm}
B_n(m)=\frac{(-1)^{m-1}}{m}\sum_{i=0}^{m-1}(-1)^iB_n(i)A_n(m-i)\,.
\end{equation}
Setting $x_k=(-1)^k/k^r$ in (\ref{Anm.Bnm}) and letting $n\to\infty$ give the recurrence of $\ze(\{\bar{r}\}_m)$:
\begin{align}
\ze(\{\bar{r}\}_m)
&=\frac{(-1)^{m-1}}{m}\bibb{\sum_{\substack{i=0\\m-i\text{ odd}}}^{m-1}
        (-1)^i\ze(\{\bar{r}\}_i)\ze(\ol{rm-ri})
    +\sum_{\substack{i=0\\m-i\text{ even}}}^{m-1}
        (-1)^i\ze(\{\bar{r}\}_i)\ze(rm-ri)}\nonumber\\
&=\frac{(-1)^{m-1}}{m}\sum_{i=0}^{m-1}
    (-1)^i\ze(\{\bar{r}\}_i)\ze(rm-ri;(-1)^{m-i})\,,\label{AZrm}
\end{align}
with $\ze(\{\bar{r}_0\}):=1$. For example, we have
\begin{align*}
&\ze(\{\bar{r}\}_2)=-\frac{1}{2}\ze(2r)+\frac{1}{2}\ze^2(\bar{r})\,,\\
&\ze(\{\bar{r}\}_3)=\frac{1}{3}\ze(\ol{3r})-\frac{1}{2}\ze(\bar{r})\ze(2r)+\frac{1}{6}\ze^3(\bar{r})\,,\\
&\ze(\{\bar{r}\}_4)=-\frac{1}{4}\ze(4r)+\frac{1}{3}\ze(\bar{r})\ze(\ol{3r})
    +\frac{1}{8}\ze^2(2r)-\frac{1}{4}\ze(2r)\ze^2(\bar{r})+\frac{1}{24}\ze^4(\bar{r})\,.
\end{align*}
By Theorems \ref{Th.AES1} and \ref{Th.AES2} as well as Eq. (\ref{AZrm}),  we  obtain the expressions of quadratic sums.

\begin{corollary}
The (alternating) quadratic Euler sums $S(i_1,i_2,q;\si_1,\si_2,\si_3)$ satisfy
\begin{align}
&S(i_1,i_2,q;\si_1,\si_2,\si_3)\nonumber\\
&\quad=\si_1\si_2\si_3\bibb{\begin{array}{l}
        \ze(q,i_1+i_2;\si_3,\si_1\si_2)+\ze(q+i_1+i_2;\si_1\si_2\si_3)\\
        \quad+\ze(q,i_1,i_2;\si_3,\si_1,\si_2)+\ze(q+i_1,i_2;\si_1\si_3,\si_2)\\
        \quad+\ze(q,i_2,i_1;\si_3,\si_2,\si_1)+\ze(q+i_2,i_1;\si_2\si_3,\si_1)
    \end{array}}\,.\label{AES.pqr}
\end{align}
Moreover, the alternating sums $S_{\bar{r}^2,\bar{r}}$ are reducible to linear sums and alternating zeta values:
\begin{equation}\label{ASrrr}
S_{\bar{r}^2,\bar{r}}=S_{\bar{r},2r}+\frac{1}{3}\{\ze(\ol{3r})-\ze^3(\bar{r})\}\,.
\end{equation}
In particular, for any odd integer $r$, the sums $S_{\bar{r}^2,\bar{r}}$ are expressible in terms of (alternating) zeta values.
\end{corollary}

\begin{proof}
It suffices to derive (\ref{ASrrr}). From (\ref{AES.pqr}), the alternating Euler sums $S_{\bar{i}_1\bar{i}_2,\bar{q}}$ satisfy
\begin{align*}
S_{\bar{i}_1\bar{i}_2,\bar{q}}&=-\ze(\bar{q},i_1+i_2)-\ze(\ol{q+i_1+i_2})\\
    &\quad-\ze(\bar{q},\bar{i}_1,\bar{i}_2)-\ze(q+i_1,\bar{i}_2)
    -\ze(\bar{q},\bar{i}_2,\bar{i}_1)-\ze(q+i_2,\bar{i}_1)\,.
\end{align*}
Then using the recurrence (\ref{AZrm}), the reflection formula
$\ze(\bar{a},b)+\ze(b,\bar{a})=\ze(\bar{a})\ze(b)-\ze(\ol{a+b})$, and the expression $S_{\bar{p},q}=-\ze(q,\bar{p})-\ze(\ol{p+q})$, we have
\[
S_{\bar{r}^2,\bar{r}}=-\ze(\ol{3r})-\ze(\bar{r},2r)-2\ze(2r,\bar{r})-2\ze(\{\bar{r}\}_3)
    =S_{\bar{r},2r}+\frac{1}{3}\{\ze(\ol{3r})-\ze^3(\bar{r})\}\,.\qedhere
\]
\end{proof}

\begin{example}
Using Eq. (\ref{ASrrr}) and doing some computation by (\ref{AZV}), (\ref{MZV.st}) and (\ref{AES.pq}) yield
\begin{align*}
S_{\bar{1}^2,\bar{1}}&=-\frac{1}{2}\ze(3)+\frac{3}{2}\ze(2)\ln(2)+\frac{1}{3}\ln^3(2)\,,\\
S_{\bar{3}^2,\bar{3}}&=-\frac{7111}{512}\ze(9)+\frac{561}{128}\ze(2)\ze(7)
    +\frac{189}{128}\ze(3)\ze(6)+\frac{315}{64}\ze(4)\ze(5)+\frac{9}{64}\ze^3(3)\,.
\end{align*}
Moreover, combining Eq. (\ref{ASrrr}) with the evaluation of $\ze(4,\bar{2})$ gives
\[
S_{\bar{2}^2,\bar{2}}=-\ze(4,\bar{2})+\frac{53}{64}\ze(6)
    =\frac{905}{96}\ze(6)-\frac{31}{4}\ze(5)\ln(2)
        -\frac{33}{16}\ze^2(3)+4\ze(\bar{5},\bar{1})\,.\hfill\qedsymbol
\]
\end{example}

Similarly to Theorem \ref{Th.srmq1}, we present the explicit expressions of $S_{r^m,\bar{q}}$, $S_{\bar{r}^m,q}$ and $S_{\bar{r}^m,\bar{q}}$.

\begin{theorem}\label{Th.AES.rmq}
The alternating Euler sums $S_{r^m,\bar{q}}$ satisfy
\[
S_{r^m,\bar{q}}=-\sum_{\theta\in\mathcal{C}_m}\binom{m}{\theta_1,\theta_2,\ldots,\theta_p}
    \{\ze(\bar{q},r\theta_1,r\theta_2,\ldots,r\theta_p)
    +\ze(\ol{q+r\theta_1},r\theta_2,\ldots,r\theta_p)\}\,,
\]
and the alternating Euler sums $S_{\bar{r}^m,q}$ and $S_{\bar{r}^m,\bar{q}}$ satisfy
\begin{align*}
S_{\bar{r}^m,q}&=(-1)^m\sum_{\theta\in\mathcal{C}_m}\binom{m}{\theta_1,\theta_2,\ldots,\theta_p}\\
    &\quad\times\bibb{
        \begin{array}{c}
            \ze(q,r\theta_1,\ldots,r\theta_p;
                1,(-1)^{\theta_1},\ldots,(-1)^{\theta_p})\\
            +\ze(q+r\theta_1,r\theta_2,\ldots,r\theta_p;
                (-1)^{\theta_1},(-1)^{\theta_2},\ldots,(-1)^{\theta_p})
        \end{array}
    }\,,\\
S_{\bar{r}^m,\bar{q}}
    &=(-1)^{m+1}\sum_{\theta\in\mathcal{C}_m}\binom{m}{\theta_1,\theta_2,\ldots,\theta_p}\\
    &\quad\times\bibb{
        \begin{array}{c}
            \ze(q,r\theta_1,\ldots,r\theta_p;
                -1,(-1)^{\theta_1},\ldots,(-1)^{\theta_p})\\
            +\ze(q+r\theta_1,r\theta_2,\ldots,r\theta_p;
                (-1)^{\theta_1+1},(-1)^{\theta_2},\ldots,(-1)^{\theta_p})
        \end{array}
    }\,,
\end{align*}
where $\theta:=(\theta_1,\theta_2,\ldots,\theta_p)$.
\end{theorem}

\begin{proof}
In the case of $S_{r^m,\bar{q}}$, there holds $\la_c=0$; then $\tilde{T}_c(\theta,\rho)=r\theta_c$ and $\check{T}_1(\theta,\rho)=\ol{q+r\theta_1}$. In the cases of $S_{\bar{r}^m,q}$ and $S_{\bar{r}^m,\bar{q}}$, there holds $\la_c=\theta_c$, and the expressions of $\tilde{T}_c(\theta,\rho)$, $\hat{T}_1(\theta,\rho)$ and $\check{T}_1(\theta,\rho)$ depend on the parity of $\theta_c$ and $\theta_1$.
\end{proof}


\section{Maple package for (alternating) Euler sums}\label{Sec.Maple}

We have developed the Maple package to evaluate the (alternating) Euler sums
\[
S_{i_1i_2\cdots i_m,q}\,,\quad
S_{i_1\cdots i_l\bar{i}_{l+1}\cdots \bar{i}_m,q}\,,\quad
S_{i_1\cdots i_l\bar{i}_{l+1}\cdots \bar{i}_m,\bar{q}}
\]
by the theory of this paper. The readers can download our Maple package from \cite{XW2017}.

Using the function
\[
\texttt{EulerSum([$i_1,\ldots,i_m,q$])}
\]
related to Theorem \ref{Th.ES.eq1}, we obtain directly the explicit formula of $S_{i_1i_2\cdots i_m,q}$ in terms of MZVs. By further running the program on MZVs provided by Petitot \cite{Peti09}, we can transform the result and obtain the explicit formula in terms of $\ze(2)$, $\ze(6,2)$, $\ze(8,2)$, $\ze(10,2)$, $\ze(8,2,1)$, $\ze(8,2,1,1)$, etc., and $\ze(s)$, for odd integer $s\geq 3$. For the Euler sums of small weights, more transformations can be applied to the explicit formulas to obtain the familiar forms. Let us take the sum $S_{1^3,9}$ for an instance. By the function, Maple gives
\begin{align*}
S_{1^3,9}&=\ze(9,3)+3\ze(9,1,2)+3\ze(9,2,1)+6\ze(9,1,1,1)\\
    &\quad+\ze(12)+3\ze(10,2)+3\ze(11,1)+6\ze(10,1,1)\\
&=\frac{1060345}{22112}\ze(12)-35\ze(3)\ze(9)-33\ze(5)\ze(7)+3\ze(2)\ze(3)\ze(7)
    +\frac{3}{2}\ze(2)\ze^2(5)\\
    &\quad+\frac{21}{4}\ze(6)\ze^2(3)+\frac{15}{2}\ze(3)\ze(4)\ze(5)
    -\frac{1}{4}\ze^4(3)+\frac{15}{4}S_{2,10}\,.
\end{align*}

It can be found that the number of non-alternating Euler sums of weight $w$ is $\sum_{k=1}^{w-2}p(k)$, where $p(k)$ is the number of partitions of the positive integer $k$. Thus, the numbers of non-alternating Euler sums of weights $3,4,5,6,7,8,9,10,11$ are $1,3,6,11,18,29,44,66,96$, respectively. The evaluations of these Euler sums were presented in \cite{FlSa98,SunPing03,SunPing07,WangLyu18.ESSS,WangXu17.ESW,
Xu17.MZVES,XuCh16.7Sum,XuLi17.TTS,XuYS16.ESI}. Using the package developed in this paper, we have checked all the non-alternating Euler sums of weight $w\leq 11$ directly.

Just as the non-alternating case, using the function
\[
\texttt{AEulerSum([$i_1,\ldots,i_m,q$])}
\]
related to Theorems \ref{Th.AES1} and \ref{Th.AES2}, we can express the alternating Euler sums in terms of alternating MZVs. Note that in this function, we use $-r$ instead of the parameter $\bar{r}$. Additionally, by substituting the evaluations of the alternating MZVs given in \cite{BBV2010}, the alternating Euler sums can be further reduced to zeta values, polylogarithms, $\ln(2)$, etc. For example, to compute $S_{1^2,\bar{3}}$, we should input $\texttt{AEulerSum([$1,1,-3$])}$, and Maple gives
\begin{align*}
S_{1^2,\bar{3}}&=-\ze(\bar{5})-2\ze(\bar{4},1)-\ze(\bar{3},2)-2\ze(\bar{3},1,1)\\
&=4\Lii_5\smbb{\frac{1}{2}}-\frac{19}{32}\ze(5)
    +4\Lii_4\smbb{\frac{1}{2}}\ln(2)-\frac{11}{8}\ze(3)\ze(2)
    +\frac{7}{4}\ze(3)\ln^2(2)\\
&\quad-\frac{2}{3}\ze(2)\ln^3(2)+\frac{2}{15}\ln^5(2)\,.
\end{align*}
Similarly, we have
\begin{align*}
S_{\bar{1}^2,3}&=4\Lii_5\smbb{\frac{1}{2}}-\frac{167}{32}\ze(5)+\frac{19}{8}\ze(4)\ln(2)
    +\frac{3}{4}\ze(3)\ze(2)+\frac{7}{4}\ze(3)\ln^2(2)\\
    &\quad+\frac{1}{3}\ze(2)\ln^3(2)-\frac{1}{30}\ln^5(2)\,,\\
S_{\bar{1}^2,\bar{3}}&=-\frac{19}{32}\ze(5)-4\Lii_4\smbb{\frac{1}{2}}\ln(2)
    +\frac{19}{8}\ze(4)\ln(2)+\frac{3}{8}\ze(3)\ze(2)
    +\ze(2)\ln^3(2)-\frac{1}{6}\ln^5(2)\,,\\
S_{1\bar{1},3}&=2\Lii_5\smbb{\frac{1}{2}}-\frac{193}{64}\ze(5)
    +4\ze(4)\ln(2)+\frac{3}{8}\ze(3)\ze(2)-\frac{7}{8}\ze(3)\ln^2(2)\\
    &\quad+\frac{1}{6}\ze(2)\ln^3(2)-\frac{1}{60}\ln^5(2)\,,\\
S_{1\bar{1},\bar{3}}&=2\Lii_5\smbb{\frac{1}{2}}-\frac{37}{16}\ze(5)
    +4\ze(4)\ln(2)-\frac{1}{8}\ze(3)\ze(2)-\frac{7}{8}\ze(3)\ln^2(2)\\
    &\quad+\frac{1}{6}\ze(2)\ln^3(2)-\frac{1}{60}\ln^5(2)\,.
\end{align*}
Using this function, we have computed all the alternating Euler sums of weight $w\leq 6$. The evaluations of some alternating Euler sums of weight $w\leq 5$ can be found in \cite{Xu18.EEHS,Xu17.SECES,XuCai18.OHNN,XuCh16.FACM,XuYZ17.EEQES}, so we list in Table \ref{Tab.Eval.AE.6}, as examples, the evaluations of $18$ alternating Euler sums of weight $6$, which are of the form
\[
S_{\pi_1\pi_2\cdots\pi_k,\bar{q}}
    =\sum_{n=1}^{\infty}(-1)^{n-1}\frac{H_n^{(\pi_1)}H_n^{(\pi_2)}\cdots H_n^{(\pi_k)}}{n^q}\,.
\]
Note that the values in Table \ref{Tab.Eval.AE.6} represent the coefficients. For instance, from this table, we have
\begin{align*}
S_{1^5,\bar{1}}&=-\frac{91}{16}\ze(6)
    -20\Lii_5\smbb{\frac{1}{2}}\ln(2)+\frac{565}{32}\ze(5)\ln(2)
    +5\Lii_4\smbb{\frac{1}{2}}\ze(2)\\
&\quad-10\Lii_4\smbb{\frac{1}{2}}\ln^2(2)-10\ze(4)\ln^2(2)
    -\frac{79}{32}\ze^2(3)+\frac{45}{8}\ze(3)\ze(2)\ln(2)\\
&\quad-\frac{15}{4}\ze(3)\ln^3(2)+\frac{55}{24}\ze(2)\ln^4(2)-\frac{5}{12}\ln^6(2)\,.
\end{align*}
Moreover, in this table, we use $\ze(\bar{5},1)$ to replace $\ze(\bar{5},\bar{1})$, because when $\ze(\bar{5},1)$ is used, more coefficients are vanish.

\renewcommand\arraystretch{1.5}

\begin{sidewaystable}[htbp]
{\scriptsize\centering
\begin{tabular}{|c|c|c|c|c|c|c|c|c|c|c|c|c|c|c|}
\hline
\!\!No.\!\!
    &\!\!Sums\!\!
    &\!\!$\Li_6(\frac{1}{2})$\!\!
    &\!\!$\ze(6)$\!\!
    &\!\!$\Li_5(\frac{1}{2})\ln(2)$\!\!
    &\!\!$\ze(5)\ln(2)$\!\!
    &\!\!$\Li_4(\frac{1}{2})\ze(2)$\!\!
    &\!\!$\Li_4(\frac{1}{2})\ln^2(2)$\!\!
    &\!\!$\ze(4)\ln^2(2)$\!\!
    &\!\!$\ze^2(3)$\!\!
    &\!\!$\ze(3)\ze(2)\ln(2)$\!\!
    &\!\!$\ze(3)\ln^3(2)$\!\!
    &\!\!$\ze(2)\ln^4(2)$\!\!
    &\!\!$\ln^6(2)$\!\!
    &\!\!$\ze(\bar{5},1)$\!\!
    \\\hline
1&$S_{1^5,\bar{1}}$
    &\!\!$0$\!\!
    &\!\!$-\frac{91}{16}$\!\!
    &\!\!$-20$\!\!
    &\!\!$\frac{565}{32}$\!\!
    &\!\!$5$\!\!
    &\!\!$-10$\!\!
    &\!\!$-10$\!\!
    &\!\!$-\frac{79}{32}$\!\!
    &\!\!$\frac{45}{8}$\!\!
    &\!\!$-\frac{15}{4}$\!\!
    &\!\!$\frac{55}{24}$\!\!
    &\!\!$-\frac{5}{12}$\!\!
    &\!\!$0$\!\!
    \\\hline
2&$S_{1^32,\bar{1}}$
    &\!\!$12$\!\!
    &\!\!$-\frac{2411}{192}$\!\!
    &\!\!$6$\!\!
    &\!\!$\frac{155}{32}$\!\!
    &\!\!$-1$\!\!
    &\!\!$3$\!\!
    &\!\!$-\frac{11}{16}$\!\!
    &\!\!$\frac{15}{8}$\!\!
    &\!\!$-\frac{3}{4}$\!\!
    &\!\!$\frac{7}{8}$\!\!
    &\!\!$-\frac{5}{12}$\!\!
    &\!\!$\frac{11}{120}$\!\!
    &\!\!$-5$\!\!
    \\\hline
3&$S_{1^23,\bar{1}}$
    &\!\!$0$\!\!
    &\!\!$\frac{91}{32}$\!\!
    &\!\!$4$\!\!
    &\!\!$-\frac{31}{4}$\!\!
    &\!\!$-1$\!\!
    &\!\!$2$\!\!
    &\!\!$\frac{59}{16}$\!\!
    &\!\!$-\frac{11}{64}$\!\!
    &\!\!$\frac{3}{8}$\!\!
    &\!\!$-\frac{1}{4}$\!\!
    &\!\!$-\frac{5}{24}$\!\!
    &\!\!$\frac{1}{20}$\!\!
    &\!\!$0$\!\!
    \\\hline
4&$S_{12^2,\bar{1}}$
    &\!\!$0$\!\!
    &\!\!$-\frac{35}{16}$\!\!
    &\!\!$-8$\!\!
    &\!\!$\frac{341}{32}$\!\!
    &\!\!$1$\!\!
    &\!\!$-4$\!\!
    &\!\!$-\frac{17}{4}$\!\!
    &\!\!$\frac{27}{32}$\!\!
    &\!\!$-\frac{7}{8}$\!\!
    &\!\!$0$\!\!
    &\!\!$\frac{3}{8}$\!\!
    &\!\!$-\frac{1}{10}$\!\!
    &\!\!$0$\!\!
    \\\hline
5&$S_{14,\bar{1}}$
    &\!\!$0$\!\!
    &\!\!$\frac{851}{192}$\!\!
    &\!\!$0$\!\!
    &\!\!$-\frac{93}{32}$\!\!
    &\!\!$-1$\!\!
    &\!\!$0$\!\!
    &\!\!$\frac{17}{16}$\!\!
    &\!\!$-\frac{9}{8}$\!\!
    &\!\!$0$\!\!
    &\!\!$0$\!\!
    &\!\!$-\frac{1}{24}$\!\!
    &\!\!$0$\!\!
    &\!\!$3$\!\!
    \\\hline
6&$S_{23,\bar{1}}$
    &\!\!$0$\!\!
    &\!\!$-\frac{647}{192}$\!\!
    &\!\!$0$\!\!
    &\!\!$\frac{31}{16}$\!\!
    &\!\!$2$\!\!
    &\!\!$0$\!\!
    &\!\!$-\frac{5}{4}$\!\!
    &\!\!$\frac{3}{4}$\!\!
    &\!\!$\frac{3}{8}$\!\!
    &\!\!$0$\!\!
    &\!\!$\frac{1}{12}$\!\!
    &\!\!$0$\!\!
    &\!\!$-2$\!\!
    \\\hline
7&$S_{5,\bar{1}}$
    &\!\!$0$\!\!
    &\!\!$\frac{111}{64}$\!\!
    &\!\!$0$\!\!
    &\!\!$-\frac{15}{16}$\!\!
    &\!\!$0$\!\!
    &\!\!$0$\!\!
    &\!\!$0$\!\!
    &\!\!$-\frac{9}{32}$\!\!
    &\!\!$0$\!\!
    &\!\!$0$\!\!
    &\!\!$0$\!\!
    &\!\!$0$\!\!
    &\!\!$0$\!\!
    \\\hline
8&$S_{1^4,\bar{2}}$
    &\!\!$-24$\!\!
    &\!\!$\frac{2159}{96}$\!\!
    &\!\!$-24$\!\!
    &\!\!$0$\!\!
    &\!\!$6$\!\!
    &\!\!$-12$\!\!
    &\!\!$-\frac{15}{4}$\!\!
    &\!\!$-\frac{195}{32}$\!\!
    &\!\!$\frac{21}{4}$\!\!
    &\!\!$-\frac{7}{2}$\!\!
    &\!\!$\frac{7}{4}$\!\!
    &\!\!$-\frac{1}{3}$\!\!
    &\!\!$10$\!\!
    \\\hline
9&$S_{1^22,\bar{2}}$
    &\!\!$8$\!\!
    &\!\!$-\frac{1409}{192}$\!\!
    &\!\!$8$\!\!
    &\!\!$0$\!\!
    &\!\!$-1$\!\!
    &\!\!$4$\!\!
    &\!\!$\frac{5}{8}$\!\!
    &\!\!$\frac{77}{64}$\!\!
    &\!\!$-\frac{7}{8}$\!\!
    &\!\!$\frac{7}{6}$\!\!
    &\!\!$-\frac{13}{24}$\!\!
    &\!\!$\frac{1}{9}$\!\!
    &\!\!$-3$\!\!
    \\\hline
10&$S_{13,\bar{2}}$
    &\!\!$0$\!\!
    &\!\!$-\frac{331}{384}$\!\!
    &\!\!$0$\!\!
    &\!\!$0$\!\!
    &\!\!$0$\!\!
    &\!\!$0$\!\!
    &\!\!$0$\!\!
    &\!\!$\frac{75}{64}$\!\!
    &\!\!$0$\!\!
    &\!\!$0$\!\!
    &\!\!$0$\!\!
    &\!\!$0$\!\!
    &\!\!$-\frac{7}{2}$\!\!
    \\\hline
11&$S_{2^2,\bar{2}}$
    &\!\!$0$\!\!
    &\!\!$\frac{521}{96}$\!\!
    &\!\!$0$\!\!
    &\!\!$0$\!\!
    &\!\!$-4$\!\!
    &\!\!$0$\!\!
    &\!\!$\frac{5}{2}$\!\!
    &\!\!$\frac{23}{16}$\!\!
    &\!\!$-\frac{7}{2}$\!\!
    &\!\!$0$\!\!
    &\!\!$-\frac{1}{6}$\!\!
    &\!\!$0$\!\!
    &\!\!$4$\!\!
    \\\hline
12&$S_{4,\bar{2}}$
    &\!\!$0$\!\!
    &\!\!$\frac{389}{192}$\!\!
    &\!\!$0$\!\!
    &\!\!$0$\!\!
    &\!\!$0$\!\!
    &\!\!$0$\!\!
    &\!\!$0$\!\!
    &\!\!$-\frac{15}{16}$\!\!
    &\!\!$0$\!\!
    &\!\!$0$\!\!
    &\!\!$0$\!\!
    &\!\!$0$\!\!
    &\!\!$4$\!\!
    \\\hline
13&$S_{1^3,\bar{3}}$
    &\!\!$-12$\!\!
    &\!\!$\frac{771}{64}$\!\!
    &\!\!$-12$\!\!
    &\!\!$0$\!\!
    &\!\!$3$\!\!
    &\!\!$-6$\!\!
    &\!\!$-\frac{15}{8}$\!\!
    &\!\!$-\frac{207}{64}$\!\!
    &\!\!$\frac{21}{8}$\!\!
    &\!\!$-\frac{7}{4}$\!\!
    &\!\!$\frac{7}{8}$\!\!
    &\!\!$-\frac{1}{6}$\!\!
    &\!\!$\frac{9}{2}$\!\!
    \\\hline
14&$S_{12,\bar{3}}$
    &\!\!$0$\!\!
    &\!\!$\frac{29}{192}$\!\!
    &\!\!$0$\!\!
    &\!\!$0$\!\!
    &\!\!$1$\!\!
    &\!\!$0$\!\!
    &\!\!$-\frac{5}{8}$\!\!
    &\!\!$-\frac{49}{64}$\!\!
    &\!\!$\frac{7}{8}$\!\!
    &\!\!$0$\!\!
    &\!\!$\frac{1}{24}$\!\!
    &\!\!$0$\!\!
    &\!\!$\frac{3}{2}$\!\!
    \\\hline
15&$S_{3,\bar{3}}$
    &\!\!$0$\!\!
    &\!\!$-\frac{113}{64}$\!\!
    &\!\!$0$\!\!
    &\!\!$0$\!\!
    &\!\!$0$\!\!
    &\!\!$0$\!\!
    &\!\!$0$\!\!
    &\!\!$\frac{63}{32}$\!\!
    &\!\!$0$\!\!
    &\!\!$0$\!\!
    &\!\!$0$\!\!
    &\!\!$0$\!\!
    &\!\!$-6$\!\!
    \\\hline
16&$S_{1^2,\bar{4}}$
    &\!\!$0$\!\!
    &\!\!$\frac{241}{192}$\!\!
    &\!\!$0$\!\!
    &\!\!$0$\!\!
    &\!\!$0$\!\!
    &\!\!$0$\!\!
    &\!\!$0$\!\!
    &\!\!$-\frac{1}{4}$\!\!
    &\!\!$0$\!\!
    &\!\!$0$\!\!
    &\!\!$0$\!\!
    &\!\!$0$\!\!
    &\!\!$-1$\!\!
    \\\hline
17&$S_{2,\bar{4}}$
    &\!\!$0$\!\!
    &\!\!$\frac{361}{192}$\!\!
    &\!\!$0$\!\!
    &\!\!$0$\!\!
    &\!\!$0$\!\!
    &\!\!$0$\!\!
    &\!\!$0$\!\!
    &\!\!$-\frac{3}{4}$\!\!
    &\!\!$0$\!\!
    &\!\!$0$\!\!
    &\!\!$0$\!\!
    &\!\!$0$\!\!
    &\!\!$4$\!\!
    \\\hline
18&$S_{1,\bar{5}}$
    &\!\!$0$\!\!
    &\!\!$\frac{31}{32}$\!\!
    &\!\!$0$\!\!
    &\!\!$0$\!\!
    &\!\!$0$\!\!
    &\!\!$0$\!\!
    &\!\!$0$\!\!
    &\!\!$0$\!\!
    &\!\!$0$\!\!
    &\!\!$0$\!\!
    &\!\!$0$\!\!
    &\!\!$0$\!\!
    &\!\!$-1$\!\!
    \\\hline
\end{tabular}
\caption{The evaluations of $18$ alternating Euler sums of weight $6$}\label{Tab.Eval.AE.6}
}
\end{sidewaystable}

The method used in the present paper reflects the quasi shuffle algebra of nested sums \cite{MoUW02.NSE,Ver99}: one could use the quasi shuffle algebra to find a linear representation of the numerator of the summand, and then split the sums. This method is implemented in symbolic computation packages like \texttt{XSummer} \cite{MoU06.XTF,MoUW02.NSE} and \texttt{HarmonicSums} \cite{Abling13,Abling14}, and is used in certain calculations in quantum chromodynamics. However, in this paper, we give the explicit formulas to expand (alternating) Euler sums in terms of (alternating) MZVs, and develop the Maple package to evaluate these sums according to the explicit formulas. To the best of our knowledge, in the literature, there is no such general explicit formula satisfied by all the (alternating) Euler sums.


\section*{Acknowledgments}

The authors would like to thank the anonymous referees for their valuable
comments and suggestions, and Professor Petitot for his kind help. The second author is supported by the National Natural Science Foundation of China (under Grant 11671360), and the Zhejiang Provincial Natural Science Foundation of China (under Grant LQ17A010010).


\end{document}